\documentclass[11pt]{article}
\usepackage[utf8]{inputenc}
\usepackage[T1]{fontenc}
\usepackage{amssymb,amsmath,amsthm}
\usepackage[english]{babel}
\usepackage{mathtools}
\usepackage{enumerate}
\usepackage{authblk}
\usepackage{graphicx}
\usepackage{subcaption}

\let \le \leqslant
\let \leq \leqslant
\let \ge \geqslant
\let \geq \geqslant
\let \epsilon \varepsilon
\let \phi \varphi

\let\OLDthebibliography\thebibliography
\renewcommand\thebibliography[1]{
  \OLDthebibliography{#1}
  \setlength{\parskip}{0pt}
  \setlength{\itemsep}{0pt plus 0.3ex}
}

\numberwithin{equation}{section}

\providecommand{\UH}{\mathbb{H}}

\providecommand{\CC}{\mathcal{C}}
\DeclareMathOperator{\dist}{dist}

\renewcommand{\epsilon}{\varepsilon}
\providecommand{\W}{\mathcal{W}}

\newtheorem{thm}{Theorem}[section]
\newtheorem{lem}[thm]{Lemma}
\newtheorem{prop}[thm]{Proposition}
\newtheorem{cor}[thm]{Corollary}
\theoremstyle{definition}
\newtheorem*{rmk}{Remark}

\title{Remarks on the regularity of quasislits}

\author[1]{Lukas Schoug}
\author[2]{Atul Shekhar}
\author[3]{Fredrik Viklund}
\affil[1,3]{\it KTH Royal Institute of Technology}
\affil[2]{\it Universit\'e Lyon 1}

\date{}

\begin{document}

\maketitle

\begin{abstract}
A quasislit is the image of a vertical line segment $[0, iy]$, $y > 0$, under a quasiconformal homeomorphism of the upper half-plane fixing $\infty$. Quasislits correspond precisely to curves generated by the Loewner equation with a driving function in the Lip-$\frac{1}{2}$ class.
    It is known that a quasislit is contained in a cone depending only on its Loewner driving function Lip-$\frac{1}{2}$ seminorm, $\sigma$. In this note we use the Loewner equation to give quantitative estimates on the opening angle of this cone in the full range $\sigma <4$. The estimate is shown to be sharp for small $\sigma$. As consequences, we derive explicit H\"older exponents for $\sigma < 4$ as well as estimates on winding rates. We also relate quantitatively the Lip-$\frac{1}{2}$ seminorm with the quasiconformal dilatation and discuss the optimal regularity of quasislits achievable through reparametrization.
\end{abstract}

\section{Introduction}
A quasicircle is the image of the unit circle under a quasiconformal homeomorphism of the complex plane.
Recall that $f \in W^{1,2}_{loc}$ is a $k$-quasiconformal map if it is a homeomorphic solution to the  Beltrami equation,
    $\overline{\partial} f = \mu(z) \partial f$,
where $\mu$ is measurable, complex valued, and such that $k:= \|\mu\|_\infty < 1$ with corresponding maximal dilatation $(1+k)/(1-k)$. Many equivalent characterizations of quasicircles exist, both analytic and geometric, see e.g. \cite{gehring}. For a particular characterization it is quite natural to ask how the geometry of the quasicircle depends quantitatively on the given data and how the data for different characterizations are related. For example, a $k$-quasicircle is the image of a $k$-quasiconformal map and a well-known theorem of Smirnov (motivated by a conjecture of Astala) states that the dimension of such a quasicircle is at most $1+k^2$ for $k$ small.\footnote{Recent work of Ivrii shows however that this bound is not sharp, see Section~\ref{smirnovseq}.}

In this note we will quantify some simple geometric features of quasicircles seen from the point of view of the Loewner equation. In this case it is more convenient to consider \emph{quasislits}, that is, the image of a line segment $\{iy: 0\leq y \leq y_0\}$ $(y_0 \in (0, \infty])$ under a quasiconformal homeomorphism $\UH \rightarrow \UH$, fixing $\infty$. Every quasislit has a Loewner driving function in the Lip-$\frac{1}{2}$ class (see below) and conversely, every continuous function is the Loewner driving function for a quasislit if it is in the Lip-$\frac{1}{2}$ class with small seminorm. We will primarily be interested in understanding how properties of the curve depend quantitatively on this seminorm. 
\subsection{Curves, the Loewner equation, quasiarcs, and Lip-$\frac{1}{2}$} 
A curve is an equivalence class of continuous functions $[0,1] \to \mathbb{C}$, where two representatives are in the same equivalence class if and only if each one can be obtained from the other by an increasing reparametrization. That is, $\gamma_1$ and $\gamma_2$ describe the same curve if and only if there exist increasing homeomorphisms $\alpha_1, \alpha_2 : [0,1] \to [0,1]$ such that $\gamma_1 \circ \alpha_1(t) = \gamma_2 \circ \alpha_2(t), \, t \in [0,1]$. We will almost always consider curves with a particular parametrization chosen and refer to this as a curve as well. 

Let $\lambda_t = \lambda(t)$, $t \in [0,1]$, be a continuous real-valued function with $\lambda_0 = 0$, and consider for $z \in \UH = \{z : \textup{Im} \, z > 0\}$ the solution $(f_t(z))$ to the  Loewner PDE
\begin{align}\label{eq:LPDE}
    \partial_t f_t(z) = -f'_t(z) \frac{2}{z-\lambda_t}, \quad f_0(z) = z.
\end{align}
The family $(f_t)$ is called a Loewner chain and for each $t$, $f_t:\UH \to H_t \subset \UH$ is a conformal map normalized `hydrodynamically' at $\infty$ by $f_t(z) = z - 2t/z + O(1/|z|^2)$. 
If $\lambda_t$ is sufficiently well-behaved, the limit
\begin{equation}\label{eq:gamma}
\gamma(t) = \lim_{y \downarrow 0} f_t(\lambda_t + iy)
\end{equation}
exists for every $t \in [0,1]$, defines a continuous function $t\mapsto \gamma(t)$, and the simply connected domain $H_t$ is the unbounded connected component of $\UH \setminus \gamma[0,t]$.  In this case, we say that the Loewner chain is generated by the curve $\gamma$. Viewed differently, the driving term $\lambda$ generates the (chordal) Loewner curve $\gamma=\gamma^\lambda$ which comes equipped with a particular parametrization from the Loewner equation via \eqref{eq:gamma} called the (half-plane) capacity parametrization. This process can be reversed, and starting from, e.g., a simple curve in $\overline{\mathbb{H}}$ meeting $\mathbb{R}$ non-tangentially at its starting point, otherwise staying in $\UH$, one can parametrize by capacity and recover its driving term from the hydrodynamically normalized uniformizing conformal maps $f_t^{-1} : \mathbb{H} \setminus \gamma[0,t] \to \mathbb{H}$.

\begin{rmk}Applying $z \mapsto z^2$, and `completing' the resulting curve in union with $\mathbb{R}_+$ by a hyperbolic geodesic from the end-point to $\infty$, we can naturally think of a Loewner curve as part of a loop in $\hat{\mathbb{C}}$ through $0$ and $\infty$ containing $\mathbb{R}_+$ as a sub-arc. In fact, any simple loop in $\hat{\mathbb{C}}$ (that does not necessarily contain $\mathbb{R}_+$ as a sub-arc) can be described by a two-sided Loewner equation, driven by a function defined on $\mathbb{R}$, see Section~6 of \cite{Wan18}.\end{rmk}

Here we will be interested in curves corresponding to driving terms in the Lip-$\frac{1}{2}$ class, that is,  functions satisfying
\begin{align*}
    \| \lambda \|_{\frac{1}{2}} := \sup_{s \neq t} \frac{|\lambda_t - \lambda_s|}{|t-s|^{\frac{1}{2}}} < \infty.
\end{align*}
For $\sigma > 0$, write
\begin{align*}
    \Lambda_\sigma = \{ \lambda: [0,1] \to \mathbb{R} \mid \| \lambda \|_{\frac{1}{2}} \leq \sigma, \lambda_0 = 0 \}.
\end{align*}
The fundamental observation is due to Marshall and Rohde \cite{MR05}: there exists $C$ so that if $\lambda \in \Lambda_\sigma$ with $\sigma < C$ then $\gamma^\lambda$ is a quasislit, in particular a simple curve. Lind showed that one can take $C=4$ and that this is sharp in the sense that for each $\sigma \geq 4$, there exists a function $\lambda \in \Lambda_\sigma$ that does not even generate a curve \cite{Li05}. Conversely, if $\gamma[0,t_0]$ is a quasislit generated by $\lambda$, then $\lambda_t$, $t \in [0,t_0]$, is in Lip-$\frac{1}{2}$.  See \cite{LMR10, LR12,  RTZ18} for more. For later reference we note that an equivalent description of a quasislit is as a \emph{quasiarc} (i.e., the image of a line segment under a quasiconformal homeomorphism of $\mathbb{C}$) in $\overline{\mathbb{H}}$ that meets $\mathbb{R}$ non-tangentially at its starting point and otherwise stays in $\mathbb{H}$.

We may now define
\[
\Gamma_\sigma = \{\gamma^\lambda \mid \, \lambda \in \Lambda_\sigma\}, \qquad \sigma < 4
\]
and in what follows, we will consider quasislits $\gamma \in \Gamma_\sigma$.
\subsection{Results}
For $\sigma <4 $, let
\begin{align*}
    \CC_{t,\sigma} \coloneqq \sup_{\gamma \in \Gamma_\sigma} \frac{|\textup{Re}(\gamma(t))|}{\textup{Im}(\gamma(t))}.
\end{align*}
We will estimate $\CC_{t,\sigma}$. In order to state the result, fix $\sigma < 4$ and consider the following equation
\begin{align}\label{eq:psigma}
    e^x = \frac{\sqrt{16-\sigma^2}}{\sqrt{16x^2-\sigma^2(x+1)^2}}.
\end{align}
The right-hand side of \eqref{eq:psigma} is defined for $x>\sigma/(4-\sigma)$. It is a continuous function which is  strictly decreasing from $\infty$ to $0$ as $x$ ranges from $\sigma/(4-\sigma)$ to $\infty$. Therefore, the equation has a unique solution, which we will denote by $p(\sigma)$. Next, define
\begin{align}\label{eq:Lsigma}
    L_\sigma = \frac{\sigma}{\sqrt{16-\sigma^2}}(1+p(\sigma))e^{p(\sigma)}.
\end{align}
We have the following.
\begin{thm}\label{Cbounds}
The following bounds hold uniformly in $t>0$:
\begin{enumerate}[(i)]
    \item If $0<\sigma<4$, then $$\CC_{t,\sigma} \leq L_\sigma.$$
    \item If $0<\sigma<\frac{8}{\pi}$, then $$\CC_{t,\sigma} \leq \frac{\pi \sigma}{\sqrt{64-\pi^2\sigma^2}}.$$
\end{enumerate}

\end{thm}
We prove $(ii)$ in Section~\ref{part2} and $(i)$ in Section~\ref{part1}.

The function $L_{\sigma}$ has the following expansion as $\sigma \to 0$: 
\begin{equation}\label{expand-L}
L_{\sigma} = \frac{1}{4}\biggl(1+ \frac{1}{\mathcal{W}(1)}\biggr)\sigma + O(\sigma^3),
\end{equation}
where $\mathcal{W}(\cdot)$ is the Lambert $W$ function, so that $L_{\sigma} \sim 0.69 \sigma$ as $\sigma \to 0$. Comparing with $(ii)$ we see that the estimate in $(i)$ is not sharp, but does provide the first non-trivial and explicit bound that holds for all $\sigma < 4$. One the other hand, by choosing the driving function $\lambda_t = \sigma \sqrt{t} \in \Lambda_\sigma$ we see that for all $t>0$
\begin{align*}
    \CC_{t,\sigma} \geq \tan\left( \frac{\pi}{2} \frac{\sigma}{\sqrt{16+\sigma^2}} \right),
\end{align*}
see Example 4.12 in \cite{Law05}. This shows that the estimate $(ii)$ of Theorem~\ref{Cbounds} is sharp for small $\sigma$, and we have
\[\CC_{t,\sigma} = \frac{\pi}{8} \sigma + O(\sigma^3).\]

Let us discuss some consequences of Theorem~\ref{Cbounds}. Suppose $\gamma \in \Gamma_\sigma$ with $\sigma <4$. Then by \cite{RTZ18}, Remark 4.2, there exists $\alpha_\sigma \in (0,\frac{1}{2}]$, depending only on $\sigma$, such that $\gamma(t)$, $t \in [0,1]$, parametrized by half-plane capacity is Hölder continuous with exponent $\alpha_\sigma$.

Time $0$ is special for Loewner curves in $\UH$ parametrized by capacity: under weak assumptions, such curves are always H\"older-$\frac{1}{2}$ at $t=0$, see, e.g., \cite{JVL11}. To achieve better regularity one can restrict attention to strictly positive times (or consider driving terms that are constant during a small time interval starting with $0$).  For this, let us fix $\epsilon > 0$. It is not hard to show that if $\sigma < 2$, then on the interval $[\epsilon, 1]$ one has $\alpha_\sigma \ge 1-\sigma^2/4$, see, e.g., \cite{RTZ18}, but this bound is not sharp and in the range $\sigma \in [2,4)$ no quantitative estimate is known. This is in contrast with the much rougher setting of SLE where the sharp H\"older exponents for the capacity parametrization are known \cite{JVL11}. In this case, the randomness helps in the analysis, see below for further discussion.

Our first corollary is a quantitative estimate on the Hölder exponent $\alpha_\sigma$ for $\sigma \in [0, 4)$. 
\begin{cor}\label{mainresult1}
If $\gamma \in \Gamma_\sigma$ with $\sigma < 4$,  then for every $\epsilon > 0$, $\gamma$ is Hölder continuous in the capacity parametrization on $[\epsilon,1]$ with exponent $\alpha$ for each $\alpha$ satisfying
\begin{align}\label{L-sigma}
    \alpha < \frac{1}{1+L_\sigma^2},
\end{align}
where $L_\sigma$ is given by \eqref{eq:Lsigma}.
\end{cor}
Our second corollary improves on the regularity estimate \eqref{L-sigma} (and the easy one $\alpha < 1-\sigma^2/4$), in the smaller range $\sigma \in (0,8/\pi)$.
\begin{cor}\label{mainresult2}
If $\gamma \in \Gamma_\sigma$ with $\sigma < \frac{8}{\pi}$, then for any $\epsilon > 0$, $\gamma$ is Hölder continuous in the capacity parametrization on $[\epsilon, 1]$ with exponent $\alpha$ for each $\alpha$ satisfying
\begin{align*}
    \alpha < 1-\frac{\pi^2 \sigma^2}{64}.
\end{align*}
\end{cor}
The proofs of these corollaries are the same, and given in Section~\ref{holder}.

Even though the cone estimate is sharp for small $\sigma$, we do not expect the exponents of Corollary~\ref{mainresult2} to be sharp, see Section~\ref{holder}. Here we simply remark that choosing $\lambda_t = \sigma \sqrt{1-t}$ produces a curve with a spiral at $t=1$ and a H\"older exponent of $1-\sigma^2/16$, so this provides an upper bound on the optimal exponent. Indeed, this can be seen from direct computation using the representation given in Proposition~3.3 of \cite{LMR10}, see, e.g., \cite{calle}. Rohde and Viklund have conjectured that this exponent is in fact the optimal one, but we currently have no proof of a matching lower bound. 

 \begin{rmk} Let $D(\sigma)$ be the maximal Hausdorff dimension of $\gamma \in \Gamma_\sigma$.  Corollary~\ref{mainresult2} immediately implies  \[D(\sigma) \le  1+ \pi^2 \sigma^2 / 64 + O(\sigma^4)\] as $\sigma \to 0$. 
 
 \end{rmk}

Theorem~\ref{Cbounds} also gives non-trivial bound on winding rates for the curve near the tip, when $\sigma$ is small, see Section~\ref{winding}.

In Section~\ref{smirnovseq} we make a few simple observations about the related problem of estimating the optimal Hölder regularity for quasislits achievable through reparametrization and the relation between the quasiconformal dilatiation parameter $k$ and the semi-norm $\sigma$. 
\subsection*{Acknowledgements}
LS acknowledges support from the Knut and Alice Wallenberg foundation. AS acknowledges support from the Gustafsson foundation. FV acknowledges support from the Knut and Alice Wallenberg foundation, the Swedish Research Council, and the Gustafsson foundation. 

We thank Yilin Wang for useful comments on an earlier version of our paper. FV would like to thank Steffen Rohde for many inspiring discussions about the topics covered in this paper.
\section{Preliminaries}\label{prel}
In this section we recall some basic properties of the Loewner differential equation and related objects which we will need in our analysis. Recall that $g_t = f_t^{-1}$ satisfies a family of ODEs,
 \begin{align}\label{eq:LDE}
     \partial_t g_t(z) = \frac{2}{g_t(z)-\lambda_t}, \quad g_0(z) = z.
 \end{align}
 The family $(g_t)$ is also called a Loewner chain.
We will frequently write expressions of the form $dZ_t = a(t)dF(t)$ for functions $a,F$, which are to be interpreted as
\begin{align*}
    \int_{t_0}^t dZ_s = \int_{t_0}^t a(s) dF(s).
\end{align*}
 Throughout, we let $\sigma < 4$ and assume that $\gamma$ is generated by $\lambda \in \Lambda_\sigma$ and that $(g_t)_{t\geq 0}$ is the corresponding Loewner chain. 
To work with $f_t$, which satisfies a PDE, it is convenient to use the reverse flow. 
For fixed $t>0$ and $s\in[0,t]$, let $\beta_s^t = \lambda_t-\lambda_{t-s}$. If $t$ is understood from the context, we instead write $\beta_s$. Then the following holds.
\begin{lem}
Fix $t>0$ and let $\hat{h}_s$ satisfy the equation
\begin{align}\label{eq:rLDE}
    d\hat{h}_s(z) = d\beta_s^t - \frac{2}{\hat{h}_s(z)}ds, \quad \hat{h}_0(z) = z,
\end{align}
for $s\in [0,t]$ and $z \in \UH$. Then
\begin{align*}
    \hat{h}_t(z) = f_t(\lambda_t+z).
\end{align*}
\end{lem}
\begin{proof}
Use \eqref{eq:LDE} to see that the family of maps $\hat{h}_s$, defined by $\hat{h}_s(z) \coloneqq g_{t-s}(f_t(z+\lambda_t))-\lambda_{t-s}$ for $0 \leq s \leq t$, satisfies \eqref{eq:rLDE} and $\hat{h}_t(z) = f_t(\lambda_t+z)$.
\end{proof}
\begin{rmk}
Note that $\hat{h}_s(z)$ depends on $t$: if we let $\tilde{h}_s$ solve \eqref{eq:rLDE} with driving function $\lambda$ but for $t_1 \neq t$ in the same way, then it does not hold in general that $\hat{h}_s(z) = \tilde{h}_s(z)$ for $s < \min(t,t_1)$.
\end{rmk}
If we write $\hat{h}_s(z) = X_s(z) + iY_s(z)$ for $z = x+iy \in \UH$, then \eqref{eq:rLDE} is equivalent to the equations
\begin{align}
    dX_s &= d\beta_s - \frac{2X_s}{X_s^2+Y_s^2} ds, \quad X_0 = x, \label{eq:Xeq} \\
    dY_s &= \frac{2Y_s}{X_s^2 + Y_s^2}, \quad Y_0 = y. \label{eq:Yeq}
\end{align}
Next, we recall the following from \cite{RTZ18}.
\begin{lem}[Theorem 3.1 and Lemma 2.1 of \cite{RTZ18}]
Suppose $\sigma < 4$. There exists a constant $c_\sigma > 0$, depending only on $\sigma$, such that for all $0 < y \le 1$ and $s \ge 0$,
\begin{align}\label{eq:Ybounds}
    \sqrt{y^2 + c_\sigma s} \leq Y_s(iy) \leq \sqrt{y^2 + 4s}.
\end{align}
Moreover,
\begin{align}\label{eq:Xbounds}
    |X_s(iy)| \leq \sup_{0 \leq r \leq s} |\beta_s - \beta_r|.
\end{align}
\end{lem}

In order to estimate $|\textup{Re} (\gamma(t))|/\textup{Im} (\gamma(t))$ we will work with the process
\begin{align*}
    W_s(z) = \frac{X_s(z)}{Y_s(z)}
\end{align*}
and note that by \eqref{eq:gamma}
\begin{align}\label{eq:Wcone}
    \frac{\textup{Re}(\gamma(t))}{\textup{Im}(\gamma(t))} = \lim_{y \rightarrow 0} W_t(iy).
\end{align}
Moreover, we can rewrite \eqref{eq:Xeq} as
\begin{align}
    dX_s = d\beta_s - \frac{2W_s^2}{W_s^2+1} \frac{1}{X_s}ds, \quad X_0 = x. \label{eq:XWeq}
\end{align}

When estimating the Hölder exponents and the winding rates for the curves, we need to estimate $|f_t'(\lambda_t+iy)|$ and $\arg f_t'(\lambda_t+iy)$, where differentiation is with respect to the spatial variable. Since $f_t'(\lambda_t+z) = \hat{h}_t'(z)$ the following formulas for $\hat{h}_s'(z)$ will be useful. For $s\in [0,t]$,
\begin{align}
    |\hat{h}_s'(z)| &= \exp\left\{ \int_0^s \frac{2(X_r^2-Y_r^2)}{(X_r^2+Y_r^2)^2}dr\right\} = \exp\left\{ \int_0^s \frac{X_r^2-Y_r^2}{X_r^2+Y_r^2}d\log Y_r\right\} \nonumber \\
    &= \exp\left\{ \int_0^s \frac{W_r^2-1}{W_r^2+1} d\log Y_r \right\}, \label{eq:hder} \\
    \arg \hat{h}_s'(z) &= -4\int_0^s \frac{X_r Y_r}{(X_r^2+Y_r^2)^2} dr = -2 \int_0^s \frac{X_r Y_r}{X_r^2 + Y_r^2} d\log Y_r \nonumber \\
    &= -2 \int_0^s \frac{W_r}{W_r^2+1} d\log Y_r. \label{eq:harg}
\end{align}
\subsection{Time reparametrization}\label{time}
By \eqref{eq:Yeq} and \eqref{eq:Ybounds}, it follows that $Y_s:[0,\infty) \rightarrow [y,\infty)$ is a strictly increasing, continuous function and hence that the function $(Y_s^2-y^2):[0,\infty) \rightarrow [0,\infty)$ is a bijection. We denote by $\theta_s$ its inverse function, that is, $Y_{\theta_s}^2-y^2 = s$ and denote the reparametrized functions by $\tilde{X} = X_{\theta_s}$, $\tilde{Y}_s = Y_{\theta_s} = \sqrt{y^2 + s}$ and $\tilde{W}_s = \tilde{X}_s/\sqrt{y^2+s}$. Then we have
\begin{align}
    d\tilde{X}_s &= d\beta_{\theta_s} - \frac{\tilde{X}_s}{2(y^2+s)} ds, \quad \tilde{X}_0 = x, \label{eq:Xtilde} \\
    d\theta_s &= \frac{1}{4}\left( \frac{\tilde{X}_s^2}{y^2+s}+1\right) ds, \quad \theta_0 = 0. \label{eq:theta}
\end{align}
The advantage of this reparametrization is that it only leaves us with one unknown: $\tilde{X}$. This makes it easier to compare two solutions with different driving functions and will play a crucial role in the proof of Theorem \ref{Cbounds}.

\section{A bound in the range $0 \le \sigma < \frac{8}{\pi}$}\label{part2}
This section is devoted to the proof of part \textit{(ii)} of Theorem \ref{Cbounds}, and we will prove the following result, which will imply it.
\begin{prop}\label{8pibound}
Let $\lambda \in \Lambda_\sigma$ with $\sigma < \frac{8}{\pi}$, fix some $t>0$ and define $\hat{h}_s(iy) = X_s(iy)+iY_s(iy)$ as the solution to \eqref{eq:rLDE}. Then, for all $s \in [0,t]$ and $y>0$, we have
\begin{align}\label{eq:8pibound}
    |W_s(iy)| = \left| \frac{X_s(iy)}{Y_s(iy)}\right| \leq \frac{\pi \sigma}{\sqrt{64-\pi^2\sigma^2}}.
\end{align}
\end{prop}
\begin{proof}
Using the product formula together with \eqref{eq:Xeq} and \eqref{eq:Yeq} to see that
\begin{align}\label{eq:prodform}
    d(X_s Y_s) = X_s dY_s + Y_s dX_s = Y_s d\beta_s.
\end{align}
The validity of \eqref{eq:prodform}, interpreted as Riemann-Stieltjes integrals, follows since the continuity and monotonicity of $Y_s$ implies that it is of bounded variation. By partial integration, we have
\begin{align}\label{eq:intprod}
    X_s Y_s = \int_0^s Y_r d\beta_r = y\beta_s + \int_0^s (\beta_s - \beta_r) dY_r,
\end{align}
since $X_0 = 0$. Employing the time reparametrization of Section \ref{time}, using $d\tilde{Y}_r = dr/(2 \sqrt{y^2+r})$, \eqref{eq:intprod} becomes
\begin{align}
    \tilde{X}_s \sqrt{y^2+s} = y \beta_{\theta_s} + \int_0^s (\beta_{\theta_s}-\beta_{\theta_r}) \frac{dr}{2\sqrt{y^2+r}}.
\end{align}
In order to prove the result, we shall bound $\tilde{W}_s = \tilde{X}_s/\tilde{Y}_s = \tilde{X}_s/\sqrt{y^2+s}$. We have by \eqref{eq:intprod}
\begin{align}\label{eq:W2}
    \tilde{W}_s = \frac{y \beta_{\theta_s}}{y^2+s} + \frac{1}{y^2+s}\int_0^s (\beta_{\theta_s}-\beta_{\theta_r}) \frac{dr}{2\sqrt{y^2+r}}
\end{align}
We write $R_s = \sup_{r \leq s} | \tilde{W}_s|$ and note that by \eqref{eq:theta} and $\| \beta \|_\frac{1}{2} \leq \sigma$,
\begin{align}\label{eq:btbound}
    |\beta_{\theta_s}| \leq \sigma |\theta_s|^{\frac{1}{2}} \leq \frac{\sigma}{2} \sqrt{\int_0^s (\tilde{W}_u^2 + 1) du} \leq \frac{\sigma\sqrt{s}}{2} \sqrt{R_s^2+1}
\end{align}
and
\begin{align}\label{eq:btintbound}
    \left| \int_0^s (\beta_{\theta_s}-\beta_{\theta_r}) \frac{dr}{2\sqrt{y^2+r}}\right| &\leq \int_0^s \sigma \sqrt{\theta_s-\theta_r} \frac{dr}{2\sqrt{y^2+r}} \nonumber \\
    &\leq \frac{\sigma}{2} \int_0^s \sqrt{\int_r^s (\tilde{W}_u^2 + 1) du} \, \frac{dr}{2\sqrt{y^2+r}} \nonumber \\
    &\leq \frac{\sigma}{2} \sqrt{R_s^2+1} \int_0^s \sqrt{s-r} \, \frac{dr}{2\sqrt{y^2+r}}.
\end{align}
Next, we observe that
\begin{align}\label{eq:sqint}
    \int_0^s \sqrt{s-r} \,\frac{dr}{2\sqrt{y^2+r}} &= (y^2+s) \int_{\frac{y}{\sqrt{y^2+s}}}^1 \sqrt{1-u^2} \, du \nonumber \\
    &= (y^2+s)\left(\frac{\pi}{4} - \int_0^{\frac{y}{\sqrt{y^2+s}}} \sqrt{1-u^2} \, du\right) \nonumber \\
    &= (y^2+s)\left(\frac{\pi}{4} - I\left(\frac{y}{\sqrt{y^2+s}}\right)\right),
\end{align}
where $I(x) = \frac{1}{2} \arcsin(x) + \frac{1}{2} x\sqrt{1-x^2}$. Putting together equations \eqref{eq:W2}-\eqref{eq:sqint}, we get
\begin{align*}
    |\tilde{W}_s| &\leq \frac{y |\beta_{\theta_s}|}{y^2+s} + \frac{1}{y^2+s}\left| \int_0^s (\beta_{\theta_s}-\beta_{\theta_r}) \frac{dr}{2\sqrt{y^2+r}}\right| \\
    &\leq \frac{\sigma \sqrt{R_s^2+1}}{2} \left( \frac{y\sqrt{s}}{y^2+s} + \frac{\pi}{4} - I\left(\frac{y}{\sqrt{y^2+s}}\right) \right).
\end{align*}
It can be checked that $I(y/\sqrt{y^2+s}) \geq y \sqrt{s}/(y^2+s)$, and hence
\begin{align*}
    |\tilde{W}_s| \leq \frac{\pi\sigma\sqrt{R_s^2+1}}{8}.
\end{align*}
This implies that $R_s \leq \pi\sigma\sqrt{R_s^2+1}/8$ and hence that
\begin{align*}
    R_s \leq \frac{\pi \sigma}{\sqrt{64-\pi^2\sigma^2}},
\end{align*}
which proves the proposition.
\end{proof}
Now part \textit{(ii)} of Theorem \ref{Cbounds} follows by \eqref{eq:Wcone}.
\begin{rmk}
At first it might seem that the statement $|W_s(iy)| \leq K$ for $y > 0$ is stronger than $\CC_{t,\sigma} \leq K$, but in fact, they are equivalent. Indeed, the curve $(f_t(\lambda_t+iy), y>0)$ is the curve generated by the driving function $\lambda_s^* = \lambda_{s \wedge t}$, which is clearly a Lip-$\frac{1}{2}$ function. Thus, the condition $\CC_{t,\sigma} \leq K$ implies that $|W_s(iy)| \leq K$. 
\end{rmk}

\section{A bound in the range $0 \le \sigma<4$}\label{part1}
This section is devoted to the proof of the following proposition, which will give part \textit{(ii)} of Theorem \ref{Cbounds}, as in the previous section.
\begin{prop}\label{4bound}
Let $\lambda \in \Lambda_\sigma$ with $\sigma < 4$, fix some $t>0$ and define $\hat{h}_s(iy) = X_s(iy)+iY_s(iy)$ as the solution to \eqref{eq:rLDE}. Then, for all $s \in [0,t]$ and $y>0$, we have
\begin{align}\label{eq:4bound}
    |W_s(iy)| = \left| \frac{X_s(iy)}{Y_s(iy)}\right| \leq L_\sigma.
\end{align}
\end{prop}
The proof of Proposition \ref{4bound} requires more work than that of Proposition \ref{8pibound}. In proving this, we shall employ the time-change of Section \ref{time} and bound $X_s$ and $\theta_s$ by comparing them to some properly chosen functions, using a version of the Grönwall inequality, which we now state and prove.

\begin{prop}[Grönwall's inequlity]\label{GW}
Let $F(t,x)$ be a bounded, continuous function, which is increasing and continuously differentiable in the variable $x$ and $F'(t,x) = \partial_x F(x)$ is bounded. Let $U$ be a continuous function on an interval $[0,T]$ such that
\begin{align*}
    U_t \leq \int_0^t F(r,U_r) dr
\end{align*}
for every $t \in [0,T]$. If $V$ is a continuous function defined on $[0,T]$, satisfying
\begin{align*}
    V_t = \int_0^t F(r,V_r) dr
\end{align*}
for $t \in [0,T]$, then $U_t \leq V_t$ for all $t \in [0,T]$.
\end{prop}
\begin{proof}
We have that
\begin{align*}
    U_t - V_t &\leq \int_0^t (F(r,U_r)-F(r,V_r)) dr \\
    &= \int_0^t \left\{ \int_0^1 F'(r,pU_r + (1-p)V_r) dp \right\} (U_r-V_r) dr \\
    &\leq C \int_0^t (U_r - V_r) dr,
\end{align*}
and hence the claim follows from the standard Grönwall inequality since $U_0-V_0 \le 0$. This also uses the assumption that $F(r, \cdot)$ is increasing.
\end{proof}
\begin{rmk}
We can actually relax the assumptions on $F$ somewhat. The following follows by the very same proof. Let $t_U^\delta = \inf\{t\in[0,T]: U_t = \delta \}$ and $t_V^\delta = \inf\{ t \in [0,T]: V_t = \delta\}$. Assume that $F(t,x)$ is increasing in $x$ and continuously differentiable in $x$ with $F'(t,x)$ bounded on $[\epsilon,\infty)$ for every $\epsilon > 0$. If $U_t$ and $V_t$ are continuous,
\begin{align*}
    U_t \leq U_0 + \int_0^t F(r,U_r) dr
\end{align*}
for $t < t_U^0$,
\begin{align*}
    V_t = V_0 + \int_0^t F(r,V_r) dr
\end{align*}
for $t < t_V^0$ and $U_0 \leq V_0$, then for each $0 <\delta < U_0$ it holds that for all $t < \min(t_U^\delta,t_V^\delta)$, $U_t \leq V_t$. Moreover, this implies that $t_U^\delta  \leq t_V^\delta$.
\end{rmk}

We now discuss the bound on $X_s$. Consider \eqref{eq:XWeq} and note that $W_s^2/(W_s^2+1) \leq 1$ for all $s$ and that $\beta_s \in \Lambda_\sigma$. Hence, it makes sense to compare $X_s$ with solutions to the equation
\begin{align}\label{eq:compDE}
    dZ_s = \sigma d\sqrt{s} - \frac{c}{Z_s} ds, \quad Z_0 = z_0,
\end{align}
We will carry out this analysis for functions $X_s$ that have strayed from $0$, and hence, by symmetry, it will only be necessary for us to consider solutions to \eqref{eq:compDE} started from $z_0 > 0$. The solution of \eqref{eq:compDE} depends heavily on the value of $c$ and we shall only need the solution for $c = \sigma^2/8$, as will be seen later.
\begin{lem}\label{ZDE}
The solution, $Z_s$, to \eqref{eq:compDE} with $c = \sigma^2/8$ and $z_0 > 0$ exists for all $s \geq 0$ and is given by
\begin{align}\label{eq:Zsol}
    Z_s = \frac{\sigma \sqrt{s}}{2} + z_0 \exp\left\{ \W \left( \frac{\sigma \sqrt{s}}{2z_0} \right) \right\},
\end{align}
where $\W$ is the Lambert $\W$ function, that is, the nonnegative solution to the equation
\begin{align*}
    \W(x) e^{\W(x)} = x.
\end{align*}
\end{lem}
\begin{proof}
First, we note that $\W(0)=0$ implies that $Z_0 = z_0$. Next, using that
\begin{align*}
    \frac{d\W}{dx}(x) = \frac{1}{1+\W(x)},
\end{align*}
we have
\begin{align*}
    \frac{dZ_s}{ds} &= \frac{\sigma}{4\sqrt{s}} + \frac{1}{1+\W\left(\frac{\sigma \sqrt{s}}{2z_0}\right)} \frac{\sigma}{4\sqrt{s}} \\
    &= \frac{\sigma}{2\sqrt{s}} - \frac{\sigma}{4\sqrt{s}} \frac{\W\left(\frac{\sigma \sqrt{s}}{2z_0}\right)}{1+\W\left(\frac{\sigma \sqrt{s}}{2z_0}\right)} \\
    &= \frac{\sigma}{2\sqrt{s}}-\frac{\sigma}{4\sqrt{s}} \frac{1}{1+\frac{2z_0}{\sigma\sqrt{s}}\exp\left(\W\left(\frac{\sigma \sqrt{s}}{2z_0}\right)\right)} \\
    &= \frac{\sigma}{2\sqrt{s}} - \frac{\sigma^2}{8} \frac{1}{\frac{\sigma \sqrt{s}}{2} + z_0 \exp\left(\W\left(\frac{\sigma \sqrt{s}}{2z_0}\right)\right)} \\
    &= \frac{\sigma}{2\sqrt{s}}-\frac{\sigma^2}{8}\frac{1}{Z_s},
\end{align*}
and thus we are done.
\end{proof}
We now turn to the bound for $\theta_s$. For $x_0 > 0$ and $\sigma < 4$ we define the following function
\begin{align}\label{eq:Heq}
    H_{x_0}(x) = \left( \frac{\sigma\sqrt{x}}{2} + x_0 \exp\left\{ \W\left( \frac{\sigma\sqrt{x}}{2x_0} \right) \right\} \right)^2 = \frac{\sigma^2 x}{4} \left( 1 + \frac{1}{\W\left(\frac{\sigma\sqrt{x}}{2x_0}\right)} \right)^2.
\end{align}
For $\kappa \in (\frac{\sigma^2}{4},4)$, let $M = M(\kappa)$ be such that
\begin{align}\label{eq:MKrel}
    \frac{\sigma^2}{4} \left( 1 + \frac{1}{\W\left(\frac{\sigma\sqrt{M}}{2x_0} \right)}\right)^2 = \kappa.
\end{align}
Since $\W$ is an increasing function, it follows that $H_{x_0}(x) \leq \kappa x$ for $x \geq M$.
\begin{lem}\label{Vbound}
Fix $x_0,y_0 > 0$ and $\sigma < 4$ and let $V_s$ denote the solution to the differential equation
\begin{align}\label{eq:Veq}
    \frac{dV}{ds} = \frac{1}{4} \frac{H_{x_0}(V_s)}{y_0^2+s}+\frac{1}{4}, \quad V_0 = 0.
\end{align}
Then, for any $\kappa \in (\frac{\sigma^2}{4},4)$,
\begin{align*}
    V_s \leq \max\left( \frac{M}{y_0^2}, \frac{1}{4-\kappa} \right)(y_0^2+s),
\end{align*}
where $M$ and $\kappa$ are related as in \eqref{eq:MKrel}.
\end{lem}
\begin{proof}
By \eqref{eq:Veq}, $V$ is continuous and strictly increasing to $\infty$. Thus $\tau = \inf\{ s>0: V_s \geq M\}$ is finite and for $s \leq \tau$, $V_s \leq M$. For $s > \tau$, we have that
\begin{align*}
    \frac{dV_s}{ds} \leq \frac{\kappa}{4} \frac{V_s}{y_0^2+s} + \frac{1}{4},
\end{align*}
since $H_{x_0}(x) \leq \kappa x$ for $x \geq M$ and thus
\begin{align*}
    V_s - V_\tau \leq \int_\tau^s \left(\frac{\kappa}{4} \frac{V_r}{y_0^2+r} +\frac{1}{4}\right)dr.
\end{align*}
Let $N_s$ be the solution to the differential equation
\begin{align*}
    \frac{dN_s}{ds} = \frac{\kappa}{4} \frac{N_s}{y_0^2+s} + \frac{1}{4},  \quad s > \tau,
\end{align*}
given $N_\tau$. Then for $s > \tau$,
\begin{align*}
    N_s = N_\tau \left( \frac{y_0^2 + s}{y_0^2+\tau} \right)^{\frac{\kappa}{4}} + \frac{1}{4-\kappa} (y_0^2+s)-\frac{1}{4-\kappa}(y_0^2+s)^{\frac{\kappa}{4}} (y_0^2 + \tau)^{1-\frac{\kappa}{4}},
\end{align*}
and by Proposition \ref{GW}, 
\begin{align*}
    V_s \leq V_\tau \left( \frac{y_0^2+s}{y_0^2+\tau} \right)^{\frac{\kappa}{4}} + \frac{1}{4-\kappa} (y_0^2+s)-\frac{1}{4-\kappa}(y_0^2+s)^{\frac{\kappa}{4}} (y_0^2 + \tau)^{1-\frac{\kappa}{4}}.
\end{align*}
Thus,
\begin{align*}
    \frac{V_s}{y_0^2+s}-\frac{1}{4-\kappa} \leq \left( \frac{y_0^2+s}{y_0^2+\tau}\right)^{1-\frac{\kappa}{4}} \left( \frac{V_\tau}{y_0^2+\tau} - \frac{1}{4-\kappa}\right).
\end{align*}
Since $y_0^2+\tau \leq y_0^2 + s$, this implies that either
\begin{align*}
    \frac{V_s}{y_0^2+s}-\frac{1}{4-\kappa} \leq 0 \quad \textup{or} \quad \frac{V_s}{y_0^2+s} \leq \frac{V_\tau}{y_0^2 + \tau}  \leq \frac{M}{y_0^2}.
\end{align*}
Thus we have that for $s > \tau$,
\begin{align*}
    V_s \leq \max \left( \frac{M}{y_0^2}, \frac{1}{4-\kappa} \right)(y_0^2+s)
\end{align*}
which, together with the bound $V_s \leq M$ for $s \leq \tau$, concludes the proof.
\end{proof}
We are now ready to prove Proposition \ref{4bound}.
\begin{proof}[Proof of Proposition \ref{4bound}]
Fix $\sigma < 4$ and $t>0$ and let $K_\sigma = \sigma/\sqrt{16-\sigma^2}$. If $|W_s| \leq K_\sigma$ for all $s \leq t$, we are done, since $K_\sigma \leq L_\sigma$. Assume the contrary and let $s_1>0$ be such that $|W_{s_1}| > K_\sigma$. By symmetry, we may assume that $W_{s_1} > 0$. Let $s_0 = \sup\{ s < s_1: W_s \leq K_\sigma\}$ and write $(\widehat{X}_s,\widehat{Y}_s) \coloneqq (X_{s_0+s},Y_{s_0+s})$ for $s \in [0,s_1-s_0]$. Then, by \eqref{eq:XWeq},
\begin{align*}
    \widehat{X}_s &= \widehat{X}_0 + \beta_{s_0+s} - \beta_{s_0} - \int_0^s \frac{2 W_{s_0+r}^2}{W_{s_0+r}^2+1} \frac{1}{\widehat{X}_r} dr \\
    &\leq \widehat{X}_0 + \sigma\sqrt{s} - \frac{\sigma^2}{8} \int_0^s \frac{1}{\widehat{X}_r} dr,
\end{align*}
since $w \mapsto w^2/(w^2+1)$ is an increasing function and $W_{s_0+r} \geq K_\sigma$ for $r \in [0,s_1-s_0]$. By the remark after Proposition \ref{GW} together with Lemma \ref{ZDE},
\begin{align}
    \widehat{X}_s \leq \frac{\sigma\sqrt{s}}{2} +  \widehat{X}_0 \exp\left\{ \W\left(\frac{\sigma\sqrt{s}}{2 \widehat{X}_0}\right) \right\}.
\end{align}
Next, we note that
\begin{align*}
    d\widehat{Y}_s^2 = \frac{4 \widehat{Y}_s^2}{\widehat{X}_s^2 + \widehat{Y}_s^2} ds, \quad \widehat{Y}_0 = Y_{s_0}.
\end{align*}
Reparametrizing as in Section \ref{time}, that is, defining $\theta_s$ as the inverse function of $\widehat{Y}_s^2-Y_{s_0}^2$, we have
\begin{align*}
    \theta_s = \frac{1}{4} \int_0^s \left( \frac{\widehat{X}_{\theta_r}^2}{Y_{s_0}^2+r} +1 \right)dr \leq \frac{1}{4} \int_0^s \left( \frac{H_{X_{s_0}}(\theta_r)}{Y_{s_0}^2+r}+1\right) dr
\end{align*}
where $H_{X_{s_0}}$ is defined as in \eqref{eq:Heq}. Thus, by Proposition \ref{GW}, $\theta_s \leq V_s$ where $V_s$ is the solution to \eqref{eq:Veq} with $(x_0,y_0) = (X_{s_0},Y_{s_0})$. By Lemma \ref{Vbound}, we have
\begin{align}
    \theta_s \leq \max\left( \frac{M}{Y_{s_0}^2}, \frac{1}{4-\kappa} \right)(Y_{s_0}^2+s),
\end{align}
where $\kappa \in (\frac{\sigma^2}{4},4)$ and $M=M(\kappa)$ is defined by \eqref{eq:MKrel}. Moreover,
\begin{align*}
    \frac{\widehat{X}_{\theta_s}^2}{\widehat{Y}_{\theta_s}^2} = \frac{\widehat{X}_{\theta_s}^2}{Y_{s_0}^2+s} \leq \frac{H_{X_{s_0}}(\theta_s)}{Y_{s_0}^2+s}.
\end{align*}
If $\theta_s \leq M$, then
\begin{align*}
    \frac{H_{X_{s_0}}(\theta_s)}{Y_{s_0}^2+s} \leq \frac{H_{X_{s_0}}(M)}{Y_{s_0}^2+s} = \frac{\kappa M}{Y_{s_0}^2+s} \leq \frac{\kappa M}{Y_{s_0}^2}.
\end{align*}
Recalling that $H_{X_{s_0}}(x) \leq \kappa x$ for $x \geq M$, we have that if $\theta_s > M$, then
\begin{align*}
    \frac{H_{X_{s_0}}(\theta_s)}{Y_{s_0}^2+s} \leq \frac{\kappa \theta_s}{Y_{s_0}^2 + s} \leq \max \left( \frac{\kappa M}{Y_{s_0}^2},\frac{\kappa}{4-\kappa}\right).
\end{align*}
Thus,
\begin{align}\label{eq:minthis}
    \frac{|X_{s_1}|}{Y_{s_1}} \leq \max\left( \frac{\sqrt{\kappa} \sqrt{M}}{Y_{s_0}}, \frac{\sqrt{\kappa}}{\sqrt{4-\kappa}} \right).
\end{align}
Finally, we minimize \eqref{eq:minthis}, as a function of $\kappa$. Note that this is achieved for $\kappa$ such that
\begin{align*}
    \frac{\kappa M(\kappa)}{Y_{s_0}^2} = \frac{\kappa}{4-\kappa}.
\end{align*}
Since $X_{s_0}/Y_{s_0} = K_\sigma$, we obtain $L_\sigma$ as defined by \eqref{eq:Lsigma}, which concludes the proof.
\end{proof}

\section{Regularity and winding rates}
\subsection{H\"older exponents} \label{holder} 
This section proves Corollaries \ref{mainresult1} and \ref{mainresult2}, that is, we estimate the Hölder exponents for the curve $\gamma \in \Gamma_\sigma$ depending on $\sigma$. The arguments given here are standard, but we choose to give short derivations here for the convenience of the reader. In this section, constants may vary between the lines, even though they are denoted in the same way.
\begin{proof}[Proof of Corollary \ref{mainresult1} and Corollary \ref{mainresult2}]
Fix $\sigma < 4$, $y \leq 1$, $\epsilon > 0$ and $\epsilon \leq t \leq 1$ and let $s \in [0,y^2]$. First, we note that
\begin{align}
    |\gamma(t+s)-\gamma(t)| \leq |\gamma(t+s) - f_{t+s}(\lambda_{t+s} + iy))| + |\gamma(t) - f_t(\lambda_t + iy)| \nonumber\\
    + |f_{t+s}(\lambda_{t+s}+iy) - f_{t+s}(\lambda_t+iy)| + |f_{t+s}(\lambda_t+iy) - f_t(\lambda_t + iy)|. \label{eq:boundthis}
\end{align}
The first two terms we bound as follows. Since $\gamma(t) = f_t(\lambda_t + i0^+)$, we have
\begin{align*}
    |\gamma(t) - f_t(\lambda_t+iy)| \leq \int_0^y |f_t'(\lambda_t+ir)|dr.
\end{align*}
Moreover, we let
\begin{align}\label{eq:msigma}
    m_\sigma = \begin{cases}
                \min\left( L_\sigma, \frac{\pi\sigma}{\sqrt{64-\pi^2\sigma^2}}\right), &\textup{if} \ \sigma < \frac{8}{\pi} \\
                L_\sigma &\textup{if} \ \sigma \geq \frac{8}{\pi}.
                \end{cases}
\end{align}
and write $\xi_\sigma = (m_\sigma^2-1)/(m_\sigma^2+1) \in (-1,1)$. By \eqref{eq:hder} and the fact that the function $x \mapsto (x^2-1)/(x^2+1)$ is increasing on $(0,\infty)$ together with Proposition \ref{8pibound} and Proposition \ref{4bound} we have that
\begin{align*}
    |f_t'(\lambda_t+ir)| = \exp\left\{ \int_0^t \frac{W_u^2-1}{W_u^2+1} d\log Y_u \right\} \leq Y_t^{\xi_\sigma} r^{-\xi_\sigma}.
\end{align*}
By \eqref{eq:Ybounds}, we have that
\begin{align*}
    Y_t^{\xi_\sigma} \leq Y_{t+s}^{\xi_\sigma} \leq C(\sigma,\epsilon),
\end{align*}
since $s,t \leq 1$ and therefore
\begin{align*}
    \int_0^y |f_t'(\lambda_t+ir)| dr \leq C(\sigma,\epsilon) \int_0^y r^{-\xi_\sigma} dr = \frac{C(\sigma,\epsilon)}{1-\xi_\sigma} y^{1-\xi_\sigma},
\end{align*}
and we have a bound on the first two terms of \eqref{eq:boundthis}. Let $I_{s,t}^y$ denote the line segment connecting $\lambda_t+iy$ to $\lambda_{t+s}+iy$. Using the distortion theorem and that $\sqrt{s} \leq y$, have that
\begin{align*}
    &|f_{t+s}(\lambda_{t+s}+iy) - f_{t+s}(\lambda_t+iy)| \leq |\lambda_{t+s}-\lambda_t| \max_{w \in I_{s,t}^y} |f_{t+s}'(w)| \\
    &\leq \sigma \sqrt{s} \max_{w \in I_{s,t}^y} |f_{t+s}'(w)| \leq C(\sigma) y | f_{t+s}'(\lambda_{t+s} +iy)|.
\end{align*}
Moreover, since
\begin{align*}
    \dist(f_t(\lambda_t+iy),K_t) \leq \int_0^y |f_t'(\lambda_t+ir)| dr,
\end{align*}
the Koebe 1/4 theorem implies that
\begin{align*}
    y |f_t'(\lambda_t + iy)| \leq 4 \int_0^y |f_t'(\lambda_t+ir)| dr,
\end{align*}
and thus that
\begin{align*}
    |f_{t+s}(\lambda_{t+s}+iy) - f_{t+s}(\lambda_t+iy)| \leq C(\sigma) \int_0^y |f_{t+s}'(\lambda_{t+s}+ir)| dr.
\end{align*}
Finally, since $s \in [0,y^2]$, Lemma 3.5 of \cite{JVL11} implies that
\begin{align*}
    |f_{t+s}(\lambda_t+iy) - f_t(\lambda_t+iy)| \leq Cy |f_t'(\lambda_t+iy)| \leq C \int_0^y |f_t'(\lambda_t+ir)| dr.
\end{align*}
Thus, by \eqref{eq:boundthis} and the above inequalities, letting $s = y^2$, we have
\begin{align*}
    |\gamma(t+s)-\gamma(t)| \leq C(\sigma,\epsilon) s^{\frac{1}{1+m_\sigma^2}},
\end{align*}
that is, $\gamma$ is Hölder continuous with exponent $\alpha$ for each 
\begin{align*}
    \alpha \le \frac{1}{1+m_\sigma^2}
\end{align*}
and this concludes the proof.
\end{proof}
\begin{rmk}We have seen that the H\"older exponent is determined by the behavior of the derivative near the tip of the curve, which in turn is estimated using \eqref{eq:hder}. The derivative can can be explicitly bounded if the reverse flow stays in a particular cone so that $\sup |W_r| \le C$. However, the sharp behavior depends on the integrated values of $(W_r^2-1)/(W_r^2+1)$ and we expect that the optimal $L^1$-bound is strictly smaller than the integrated optimal $L^\infty$-bound. In the case of SLE, there is of course no almost sure $L^\infty$-bound, but the process $W_r$ has an invariant distribution in an appropriately weighted measure, and \eqref{eq:hder} can be precisely estimated using this invariant distribution and an intermediate deviations argument, see \cite{JVL11}. In the case of a general Lip-$\frac{1}{2}$ function these techniques are not available. \end{rmk}

\subsection{Winding rates}\label{winding}
Theorem~\ref{Cbounds} also easily implies estimates on the winding rate at the tip of the curve $\gamma \in \Gamma_\sigma$ depending on $\sigma$, that is, we obtain estimates on the growth rate of $|\arg f_t'(\lambda_t+iy) \, |$ as $y$ tends to $0$. Geometrically, this measures the winding of the hyperbolic geodesic from $\gamma(t)$ to $\infty$ in $\UH \setminus \gamma([0,t])$, close to the tip $\gamma(t)$ when the geodesic is parametrized by harmonic measure. By \eqref{eq:harg} and \eqref{eq:Ybounds} we have a trivial bound
\begin{align*}
    |\arg f_t'(\lambda_t+iy) \, | \leq 2\log \frac{Y_t}{y} \leq \log(y^2+4t) + 2\log (y^{-1}),
\end{align*}
since $2|W_s|/(W_s^2+1) \leq 1$.
By the easy estimate $\CC_{t,\sigma} \leq \sigma/\sqrt{4-\sigma^2}$ of \cite{RTZ18}, we get a non-trivial bound in the case $\sigma < \sqrt{2}$. By virtue of Corollary \ref{mainresult2}, we have the following improvement for $\sigma < 4\sqrt{2}/\pi$.
\begin{prop}
Let $\gamma \in \Gamma_\sigma$ with $\sigma < 4\sqrt{2}/\pi$. Then
\begin{align*}
    |\arg f_t'(\lambda_t+iy) \, | \leq \frac{\pi \sigma \sqrt{64-\pi^2\sigma^2}}{64} \log(y^2 + 4t)+\frac{\pi \sigma\sqrt{64-\pi^2\sigma^2}}{32} \log y^{-1}
\end{align*}
\end{prop}
\begin{proof}
By \eqref{eq:8pibound} and since $\sigma < 4\sqrt{2}/\pi$, we have that
\begin{align*}
    |W_s(iy)| \leq \frac{\pi\sigma}{\sqrt{64-\pi^2\sigma^2}} < 1.
\end{align*}
By \eqref{eq:harg}, Proposition \ref{8pibound}, Proposition \ref{4bound}, \eqref{eq:Ybounds} and that the function $x \mapsto x/(x^2+1)$ is increasing on $[0,1]$, we have that
\begin{align*}
    |\arg f_t'(\lambda_t+iy) \, | &\leq 2 \int_0^t \frac{|W_r|}{W_r^2+1} d\log Y_r \leq 2 \frac{\pi \sigma \sqrt{64-\pi^2\sigma^2}}{64} \log\left( \frac{Y_t}{y} \right) \\
    &\leq \frac{\pi \sigma \sqrt{64-\pi^2\sigma^2}}{64} \log(y^2 + 4t)+\frac{\pi \sigma \sqrt{64-\pi^2\sigma^2}}{32} \log (y^{-1})
\end{align*}
which is the desired estimate.
\end{proof}

\section{Additional remarks}\label{smirnovseq}
In this section we collect a few simple observations that follow essentially directly from known results. We also speculate about the optimal bounds for the various exponents.

For a curve $\gamma$, consider the optimal H\"older exponent achievable through reparametrization:
\[\hat{\alpha}(\gamma)=\sup \{\alpha: \gamma \text{ can be reparametrized to be H\"older-} \alpha \}.
\]
By Corollary~\ref{mainresult2}, for small $\sigma$, we know that \[\hat{\alpha}(\Gamma_\sigma):=\inf\{\hat{\alpha}(\gamma) : \gamma \in \Gamma_\sigma\} \ge 1-\pi^2 \sigma^2/64\]
which, as remarked, trivially gives an upper bound on the maximal dimension for $\gamma \in \Gamma_\sigma$.
However, under weak regularity assumptions, satisfied by $\gamma \in \Gamma_\sigma$, the optimal estimate is equal to the maximal reciprocal Minkowski dimension. Recall that a $k$-quasiarc is the image of a line segment under a $k$-quasiconformal homeomorphism of $\mathbb{C}$.
\begin{prop}\label{smirnov}
Suppose $\gamma$ is a quasiarc with Minkowski dimension $d_M$. Then 
\begin{equation}\label{AB}
\hat{\alpha}(\gamma) = d_M^{-1}.
\end{equation}
In particular, if $\gamma$ is a $k$-quasiarc, then
\[
\hat{\alpha}(\gamma) \ge \frac{1}{1+k^2}.
\]
\end{prop}
\begin{rmk}
The second statement uses Smirnov's result \cite{Smi10}, but in fact, for small $k$, Ivrii's recent stronger result \cite{Ivr15} stating that the dimension of a $k$-quasiarc is $1+\Sigma^2 k^2 + O(k^{8/3-\epsilon})$ for small $k$, combined with Hedenmalm's estimate $\Sigma < 1$ \cite{Hed16} gives a better bound for small $k$.
\end{rmk}
\begin{proof}[Sketch of proof.]
The first statement follows from the ideas of Section~2 of \cite{AB98}. The second assertion follows from the first, combined with Smirnov's result on the maximal Hausdorff dimension of a $k$-quasiarc, and the fact, due to Astala (see Theorem~1.5 of \cite{A88}), that the maximal Hausdorff dimension over the set of $k$-quasiarcs equals the corresponding maximal Minkowski dimension. (This is not true in general for a fixed quasiarc.) Let us sketch the proof of \eqref{AB} following \cite{AB98}. For a curve $\gamma$ and $\delta > 0$, let $M(\gamma, \delta)$ be the minimal number of curve segments of diameter at most $\delta$ needed to cover $\gamma$. We first claim that if $M(\gamma, \delta) \lesssim \delta^{-r}$ then for any $r' > r$ there is a H\"older-$1/r'$ parametrization of $\gamma$. Indeed, choosing first any parametrization, $\gamma(s), \, s \in [0,1]$, not constant on any subinterval, we may define the following reparametrization 
\[
\tau(s) = \frac{\sum_n n^{-2} M(\gamma[0,s], 2^{-n}) 2^{-rn} }{\sum_n n^{-2} M(\gamma, 2^{-n}) 2^{-rn} }.
\]
Note that the denominator is bounded. Suppose $\epsilon > 0$ and $| \gamma(s_1) - \gamma(s_2)| \ge \epsilon$. If $2^{-n} < \epsilon$ and sufficiently small, then we have $M(\gamma[0,s_2], 2^{-n}) - M(\gamma[0,s_1], 2^{-n}) \ge 1$. Hence for all $\epsilon > 0$ small enough,
\[
\tau(s_2)-\tau(s_1) \ge C \sum_{n > -\log \epsilon} n^{-2} 2^{-rn} \ge C' \epsilon^{r'}.
\]
On the other hand, we can estimate $M$ in terms of $d_M$ as follows: By the quasiarc property there is a constant $C_\gamma < \infty$ such that the following holds. Given $\delta > 0$, partition $\gamma$ into segments of diameter $C_\gamma \delta$ by considering stopping times defined as follows: $t_0=0$ and then for $j = 1,2,\ldots,N_\delta$, $t_j = \inf\{ t \ge t_{j-1}: |\gamma(t) - \gamma(t_{j-1})| \ge C_\gamma \delta \}$ (terminating if the end point of the curve is within distance $C_\gamma \delta$). Then given any cover of $\gamma$ by balls of diameter $\delta$, by the quasicircle property, any ball in this cover contains at most one of the partitioning points, $\{\gamma(t_j)\}$. Hence for any $r > d_M$, we obtain $M(\gamma, \delta) \le N_\delta \lesssim \delta^{-r}$.
\end{proof}
Given Proposition~\ref{smirnov}, we would like to relate $k$ and $\sigma$ quantitatively. Marshall and Rohde \cite{MR05} show that these parameters are quantitatively related in the sense that $k \to 0$ as $\sigma \to 0$. In the other direction, taking $\lambda(t) = \sigma \sqrt{t}$ shows that one can (of course) not say anything for general quasiarcs, but if the dilatation of the quasiconformal homeomorphism of $\mathbb{H}$ defining a quasislit tends to $1$ (so that $k \to 0$), it does follow that $\sigma \to 0$. However, beyond these observations, no estimates appear to be available.  

As was pointed out to us by Rohde, a recent result of Tran can be used to get an estimate for small $\sigma$: Fix $\lambda \in \Lambda_\sigma$ with $\sigma < 1/3$. Set $\tilde \lambda = \lambda/(3 \|\lambda\|_{1/2})$ so that $\|\tilde \lambda\|_{1/2}=1/3$. On the other hand, since $\tilde{\lambda} \in \Lambda_{1/3}$ we can use Theorem~1.3 of \cite{Tran_holomorphic} to embed $\gamma^{\lambda}$ in a holomorphic motion compatible with the Loewner equation in the following sense. For each $\tau \in \mathbb{D}$, it is possible to solve the Loewner equation \eqref{eq:LPDE} with the complex driving term, $t \mapsto \tau \tilde{\lambda}_t$. The solution extends to a conformal map $g_t : \hat{\mathbb{C}} \setminus L_t \to \hat{\mathbb{C}} \setminus R_t$, where $L_t, R_t$ are both simple curves. The curve $L_t=L_t^\tau$ moves holomorphically with $\tau$, and so if we write $S=[0,2i]$, then there exists a 
 holomorphic motion of $S$, $f: \mathbb{D} \times  S  \to \mathbb{C}$, such that $f(3 \|\lambda\|_{1/2}, S) = \gamma^{\lambda}$ (as a set). By the generalized $\lambda$-lemma, $f$ extends to a holomorphic motion of $ \mathbb{C}$, $F :  \mathbb{D} \times  \mathbb{C} \to  \mathbb{C}$, and for each $\tau \in \mathbb{D}$, $F(\tau, \cdot):  \mathbb{C} \to  \mathbb{C}$ is a quasiconformal homeomorphism of dilatation at most $(1+|\tau|)/(1-|\tau|)$. Setting $\tau = 3 \|\lambda\|_{1/2} $ we conclude the following.

\begin{prop}\label{huy}
Suppose $\gamma \in \Gamma_\sigma$ with $\sigma < 1/3$. Then $\gamma$ is a $3\sigma$-quasiarc.
\end{prop}

\begin{rmk}
The condition that $\tilde \lambda \in \Lambda_{1/3}$ in \cite{Tran_holomorphic} is known to not be optimal and indeed it seems reasonable to believe that any Lip-$\frac{1}{2}$ function with seminorm strictly bounded by $4$ can similarly be embedded in a holomorphic motion. Assuming this, the argument of Proposition~\ref{huy} would show that any $\gamma \in \Gamma_\sigma$ is a $\sigma/4$-quasiarc, which, using Ivrii's asymptotics, gives the optimal regularity exponent bound $1/(1+\Sigma^2 \sigma^2/16 + o(\sigma^2))$ for small $\sigma$, slightly larger than the conjectured optimal exponent for the capacity parametrization, $1-\sigma^2/16$, which, incidentally, is what Smirnov's estimate would give.   
\end{rmk}

\end{document}